\documentclass[12pt,twoside,english]{article}
\usepackage{amsfonts,babel,amssymb,amsmath,amsthm}
\usepackage{graphicx}

\usepackage{tikz}
\newcommand*\circled[1]{\tikz[baseline=(char.base)]{
            \node[shape=circle,draw,inner sep=0.5pt] (char) {\tiny\rm#1};}}

\newtheorem{proposition}{Proposition}[section]
\newtheorem{corollary}[proposition]{Corollary}
\newtheorem{lemma}[proposition]{Lemma}
\newtheorem{theorem}[proposition]{Theorem}

\numberwithin{equation}{section}

\newcommand{\punto}{\,\,\cdot\,\,}
\newcommand{\ds}{\displaystyle}
\newcommand{\smallfrac}[2]{{\textstyle\frac{#1}{#2}}}
\newcommand{\jump}[1]{[\![#1]\!]}
\newcommand{\ave}[1]{\{\!\!\{#1\}\!\!\}}
\newcommand{\triple}[1]{|\!|\!|#1|\!|\!|}

\title{Analysis of a non-symmetric coupling of Interior Penalty DG and BEM}

\author{Norbert Heuer\footnote{Facultad de Matem\'aticas, Pontiﬁcia Universidad Cat\'olica de Chile, Avenida Vicu\~na Mackenna 4860, Santiago,
Chile. -- {\tt nheuer@mat.puc.cl}. Partially supported by FONDECYT project 1110324.} \& Francisco-Javier Sayas\footnote{Department of Mathematical Sciences, University of Delaware, Newark, DE 19716, USA --
{\tt fjsayas@math.udel.edu}}}

\date{\today}

\begin{document}

\maketitle

\begin{abstract}
We analyze a non-symmetric coupling of interior penalty discontinuous Galerkin and boundary element methods in two and three dimensions. Main results are discrete coercivity of the method, and thus unique solvability, and quasi-optimal convergence. The proof of coercivity is based on a localized variant of the variational technique from [F.-J. Sayas, The validity of Johnson-N\'edel\'ec's BEM-FEM coupling on polygonal interfaces, {\em SIAM J. Numer. Anal.}, 47(5):3451--3463, 2009]. This localization gives rise to terms which are carefully analyzed in fractional order Sobolev spaces, and by using scaling arguments for rigid transformations.
Numerical evidence of the proven convergence properties has been published previously.\\
\
\end{abstract}

\section{Introduction}

In a recent article, Of, Rodin, Steinbach and Taus \cite{OfRodSteTauSB} propose three discretization methods that combine Interior Penalty Discontinuous Galerkin methods (IPDG) with Boundary Element Methods (BEM). One of the methods falls into the category of non-symmetric coupling of Finite and Boundary Elements, while the other two belong to the general symmetric coupling philosophy. Only one of the symmetric methods is analyzed but numerical evidence of the good properties of the non-symmetric coupling is given. In this paper we prove that the non-symmetric method in \cite{OfRodSteTauSB} converges, this being, to the best of our knowledge, {\em the first successful analysis of a non-symmetric coupling of DG and BEM.}

Let us first briefly revise the milestones of the literature of BEM-FEM coupled schemes. The mathematical literature on BEM-FEM coupling can be traced back to the seminal work of Brezzi, Johnson and N\'ed\'elec \cite{BreJoh79, JohNed80}, the article \cite{JohNed80} being an early and very relevant contribution to the subject. (It has to be noted, though, that the engineering literature had previously visited these ideas and produced interesting results.) The method of \cite{JohNed80} uses one integral equation --derived from Green's Third Identity-- to construct a non-local boundary condition in order to cut off the computational domain for an exterior diffusion problem. Because of the way the analysis was approached, using compactness arguments, the cut-off boundary where the non-local integral condition is imposed had to be taken to be smooth in order for the analytical arguments to be meaningful. While practitioners never found clear reasons to dismiss the use of simpler polygonal interfaces, theory did not move very much from this initial stagnation for more than two decades. Instead, symmetric coupling methods overcame the theoretical difficulty by using a second integral equation in the set, either mixing the integral operators with the interior formulation (two-field method) or by using an additional boundary unknown (three-field method).This restored symmetry to the coupled formulation and, with it, coercivity. The original work of Costabel and Stephan \cite{Cos87, CosSte87}, and Han \cite{Han90} started this fruitful trend that for many years enjoyed the prestige of full theoretical justification. In particular, most coupling schemes for the Boundary Element method with other domain methods (Mixed Finite Elements, Non-conforming FEM or Discontinuous Galerkin approximations) were directly based on the ideas of symmetric coupling. The monograph \cite{GatHsi95} collects much of what was known in the mid-nineties on the mathematics of BEM-FEM coupling.

Non-symmetric coupling was revisited recently with the very simple result that variational techniques were enough to prove stability of Galerkin methods for non-symmetric BEM-FEM coupling. The first result in this direction \cite{Say09} was further simplified in \cite{Ste11} and
\cite{GatHsiSayTA}, and right now, it is clear that the proof itself does not contain any theoretical ingredient that does not belong to the analytical toolbox that is employed for analysis of elliptic boundary integral equations \cite{McL00}. Similar ideas are treated in \cite{MedSaySel11} for non-symmetric coupling of Mixed-FEM with BEM. As already mentioned, this work is the first contribution to non-symmetric coupling of DG and BEM.

Mathematical theory for the coupling of Discontinuous Galerkin methods (DG) and BEM is less than one decade old. The coupling of Locally DG (LDG) methods with BEM was proposed in \cite{GatSay06},
and extended in \cite{BusGatSay09}, based on a symmetric three-field formulation with an additional mortar variable. Although the LDG method has the aspect of a mixed method (it approximates both the potential and its gradient), it can be described using a so-called primal form (that uses only the potential). Barring technical difficulties, this meant that the symmetric coupling led to an analysis based on an energy (coercivity) estimate in the proper discrete norm. The primal formulation in \cite{GatSay06, BusGatSay09} is non-consistent and a Strang-type analysis is needed. A later paper
\cite{GatHeuSay10} eliminated the need of the mortar variable and many of the theoretical difficulties by demanding that the discontinuous piecewise polynomial functions that approximate the potential in the LDG method become continuous at the coupling interface. In practice this can be enforced using Lagrange multipliers (see some further explanations in  \cite{OfRodSteTauSB}). The paper \cite{GatHeuSay10} also showed how to generalize to some methods of the IPDG class.

Insisting in the symmetric approach, the paper \cite{CocSayTA} proposes a systematic approach to couple BEM with DG-FEM  and shows that most known methods fit into a double general framework. The symmetric coupling of Hybridizable DG (HDG) with BEM is proposed in \cite{CocSayTA}, although analysis was postponed to the more recent \cite{CocGuzSaySB}. Note that, in the very different context of the transient wave equation, and including only a stability analysis based on energy arguments, Abboud, Joly, Rodr\'iguez and Terrase \cite{AbbJolRodTer11} present the first time-domain coupling of DG methods with BEM. 

Coming back to the idea of non-symmetric methods, in this paper we analyze a non-symmetric coupling of IPDG with BEM proposed in \cite{OfRodSteTauSB}. The model problem will be a transmission problem in free space (in two and three dimensions). Without giving full details at this moment, let us explain what the difficulties are. The discrete scheme is the search for $(u_h,\lambda_h)\in V_h\times \Lambda_h$ such that:
\begin{equation}\label{eq:1.1}
\left[ 
\begin{array}{lll}
\ds a_{\mathrm{DG}}(u_h,v_h) -\langle \lambda_h,v_h\rangle_\Gamma 
&\ds=(f,v_h)_{\Omega_-}+\langle\beta_1, v_h\rangle_\Gamma &  \forall v_h\in V_h,\\
\ds \langle \mu_h,\smallfrac12 u_h-\mathcal K u_h\rangle_\Gamma +\langle \mu_h,\mathcal V \lambda_h\rangle_\Gamma &=\langle\mu_h,\smallfrac12\beta_0-\mathcal K\beta_0\rangle_\Gamma & \forall \mu_h\in \Lambda_h.
\end{array}
\right.
\end{equation}
Here $V_h$ is a space of discontinuous piecewise polynomial functions on a triangulation of the domain $\Omega_-$ and $\Lambda_h$ is a space of piecewise polynomials of the same degree on the triangulation of the interface $\Gamma$ that is inherited from the one of $\Omega_-$. The bilinear form $a_{\mathrm{DG}}(u_h,v_h)$ approximates the Dirichlet form $(\nabla u,\nabla v)_{\Omega_-}$ and contains a first group of non-conforming terms, including jumps of the solution and the test function and a penalization term on the element faces included to stabilize the method. The class of methods that we treat in this general DG scheme include the original Interior Penalty method of Arnold \cite{Arn82} and some symmetric and non--symmetric variants, all of them fitting in the successful unified framework developed by Arnold, Brezzi, Cockburn and Marini in \cite{ArnBreCocMar01}. The symbols $\mathcal K$ and $\mathcal V$ correspond to two boundary integral operators and create the non-local integral boundary condition for the problem. The off-diagonal terms in \eqref{eq:1.1}, involving boundary and interior quantities, are actually non-conforming approximations of $H^{-1/2}(\Gamma)\times H^{1/2}(\Gamma)$ duality products in the variational formulation of the problem.

As already intuited in a remark in \cite{OfRodSteTauSB}, the variational technique of \cite{Say09} (essentially integration by parts after recognizing that the exterior potential has to be transmitted to the interior domain) plays a key role in this analysis. However, integration by parts has now to be applied element-by-element and a whole new array of terms have to be bounded below and hidden by carefully tuned weighted Young inequalities in the essential coercivity estimate (given in Theorem \ref{th:2.1}). The analysis requires the handling of fractional order Sobolev spaces on the interior faces (scalability of these norms --or lack thereof-- will be a fact to keep in mind) and some scaling arguments that have to be dealt with rigidly, due to the fact that harmonicity of the exterior potential that is transmitted to the interior domain is needed in key steps of the process. A new source of theoretical complications stems from the fact that the discrete norm for which the coercivity result holds does not contain any term involving the value of $u_h$ on the interface $\Gamma$. This fact will make the bilinear form not bounded in the coercivity norm and a stronger norm has to be produced in upper bounds for the global bilinear form. A final detail that is not entirely obvious arises from the fact that we are dealing with a transmission problem in free space and that constants in the interior domain play a certain separate role from the analytical point of view. While this difficulty could be easily removed by considering a simpler problem (an exterior Dirichlet problem), we find it worthwhile to work out all the details for this case. The entire analysis (especially discrete coercivity) requires the careful handling of inequalities related to fractional order Sobolev norms in some reference configurations. We have striven to make all details as transparent as possible to the reader, in the hope that they will help for possible future generalizations.

The paper is structured as follows. In Section \ref{sec:2} we present the model problem, its variational formulation as a boundary-field problem and its discretization with an IPDG-BEM scheme. We also state the two main results of this paper: discrete coercivity (Theorem \ref{th:2.1}) and optimal convergence (Theorem \ref{th:2.3}). Section \ref{sec:3} is devoted to proving Theorem \ref{th:2.1}, building up from technical estimates to a full detailed proof of the discrete coercivity of the method. Section \ref{sec:4} works in a very similar way to provide a proof of Theorem \ref{th:2.3}. We remark that numerical evidence of the performance of this method has already been presented in \cite{OfRodSteTauSB}.

\paragraph{Prerequisites.} Some basic knowledge of the basic Sobolev spaces $H^m(\mathcal O)$, their norms $\|\cdot\|_{m,\mathcal O}$ and seminorms $|\punto|_{m,\mathcal O}$ will be assumed throughout. A simple subscripted norm $\|\punto\|_B$ will always refer to the $L^2(B)$ norm. The fractional order Sobolev spaces $H^{\pm1/2}(\Gamma)$ will be used from the beginning of the paper in order to introduce the formulation and to present the results. However, in Section \ref{sec:3} we will be very precise in positive and negative order Sobolev spaces both on a domain, its boundary or part of it. A very detailed reference for these results is the monograph of McLean \cite{McL00} that also includes proofs of all the mapping properties of the potentials and integral operators that will be loosely used in this work.

\section{A non-symmetric coupling of IPDG and BEM}\label{sec:2}

\subsection{Model problem and two-field non-symmetric formulation}\label{sec:2.1}

Let $\Omega_-$ be a bounded polygonal domain in the plane or a polyhedral domain with Lipschitz boundary in the space. Let $\Gamma:=\partial\Omega_-$ and $\Omega_+:=\mathbb R^d\setminus\overline{\Omega_-}$. For the sake of simplicity, we will assume that $\Omega_+$ is connected. The symbol $\gamma$ will be used to denote the trace operator, while $\partial_\nu$ will be used for the normal derivative. The model problem is a transmission problem of the form
\begin{equation}\label{eq:2.1}
\left[
\begin{array}{ll}
-\Delta u = f & \mbox{in $\Omega_-$},\\
\gamma u = \gamma u_++\beta_0 & \mbox{on $\Gamma$},\\
\partial_\nu u = \partial_\nu u_++\beta_1 & \mbox{on $\Gamma$},\\
-\Delta u_+ = 0 & \mbox{in $\Omega_+$},\\
u=\mathcal O(1/r) & \mbox{as $r\to \infty$}.
\end{array}
\right.
\end{equation}
We assume that $\beta_0\in H^{1/2}(\Gamma)$, $\beta_1\in L^2(\Gamma)$ and $f\in L^2(\Omega_-)$. A necessary and sufficient condition for existence of solution in the two-dimensional case is
\begin{equation}\label{eq:2.2}
\int_{\Omega_-} f +\int_\Gamma \beta_1 =0.
\end{equation}
In the two-dimensional case, it follows from well-known results on potential theory that
\begin{equation}\label{eq:2.3}
\partial_\nu u_+\in H^{-1/2}_0(\Gamma) :=\{\lambda \in H^{-1/2}(\Gamma)\,:\, \langle \lambda,1\rangle_\Gamma =0\},
\end{equation}
where the angled brackets are used to denote the $H^{-1/2}(\Gamma)\times H^{1/2}(\Gamma)$ duality product.

For representation of the exterior solution we need to introduce the layer potentials on $\Gamma$ and some associated boundary integral operators. Let
\[
\Phi(\mathbf x,\mathbf y):=\left\{ \begin{array}{ll}
\ds-\smallfrac1{2\pi}\log|\mathbf x-\mathbf y|, & \mbox{if $d=2$},\\
\ds\frac1{4\pi|\mathbf x-\mathbf y|},& \mbox{if $d=3$},
\end{array}\right.
\]
be the fundamental solution of the Laplace operator. Let then
\begin{eqnarray*}
\mathcal S\lambda &:=& \int_\Gamma \Phi(\punto,\mathbf y)\lambda(\mathbf y)\mathrm d\Gamma(\mathbf y),\\
\mathcal D\varphi &:=& \int_\Gamma \partial_{\nu(\mathbf y)}\Phi(\punto,\mathbf y) \varphi(\mathbf y)\mathrm d\Gamma(\mathbf y),
\end{eqnarray*}
be the single and double layer potentials, that define solutions of the Laplace equation in $\mathbb R^d\setminus\Gamma$ for arbitrary $\lambda\in H^{-1/2}(\Gamma)$ and $\varphi\in H^{1/2}(\Gamma)$. Decay at infinity of $\mathcal S\lambda$ and $\mathcal D\varphi$ is $\mathcal O(1/r)$, with the additional assumption that $\lambda\in H^{-1/2}_0(\Gamma)$ in the two-dimensional case. The exterior solution can be represented using Green's Formula (Green's Third Identity):
\begin{equation}\label{eq:2.4}
u_+=\mathcal D\gamma u_+-\mathcal S \partial_\nu u_+=\mathcal D(\gamma u-\beta_0)-\mathcal S \lambda, \qquad \lambda:=\partial_\nu u_+.
\end{equation}
Considering the boundary integral operators
\begin{equation}\label{eq:2.5}
\mathcal V \lambda:=\gamma^\pm (\mathcal S \lambda) \qquad \mathcal K\varphi:=\smallfrac12 (\gamma^+\mathcal D \varphi+\gamma^-\mathcal D \varphi),
\end{equation}
the transmission problem can be equivalently written as the search for $(u,\lambda)\in H^1(\Omega_-)\times H^{-1/2}(\Gamma)$ that satisfy
\begin{equation}\label{eq:2.6}
\left[ 
\begin{array}{lll}
\ds (\nabla u,\nabla v)_{\Omega_-} -\langle \lambda,\gamma v\rangle_\Gamma 
&\ds=(f,v)_{\Omega_-}+\langle\beta_1,\gamma v\rangle_\Gamma &  \forall v \in H^1(\Omega_-),\\
\ds \langle \mu,\smallfrac12 \gamma u-\mathcal K\gamma u\rangle_\Gamma +\langle \mu,\mathcal V \lambda\rangle_\Gamma &=\langle\mu,\smallfrac12\beta_0-\mathcal K\beta_0\rangle_\Gamma & \forall \mu \in H^{-1/2}(\Gamma),
\end{array}
\right.
\end{equation}
followed by the integral representation \eqref{eq:2.4}. Note that in the two-dimensional case, condition \eqref{eq:2.2} guarantees that $\lambda\in H^{-1/2}_0(\Gamma)$, as can be seen by testing the first equation in \eqref{eq:2.6} with $v\equiv 1$.  An important aspect of the single layer operator $\mathcal V$ is the fact that it is coercive, namely, there exists $C_\Gamma>0$ such that
\begin{equation}\label{eq:3.14}
C_\Gamma^{-1}\|\lambda\|_{-1/2,\Gamma}^2\le \langle \lambda,\mathcal V\lambda\rangle_\Gamma = \|\nabla (\mathcal S\lambda)\|_{\mathbb R^d}^2 \qquad \left\{\begin{array}{ll}\forall \lambda \in H^{-1/2}_0(\Gamma) & \mbox{if $d=2$},\\
\forall \lambda \in H^{-1/2}(\Gamma) & \mbox{if $d=3$}.\end{array}\right.
\end{equation}
In some forthcoming arguments (but not in the numerical method itself) it will be useful to decompose the interior unknown as 
\begin{equation}\label{eq:2.7}
u=u_\star+c\qquad c\in \mathbb P_0(\Omega_-)\qquad u_\star\in L^2_\star(\Omega_-):=\{ v \in L^2(\Omega_-)\,:\,(v,1)_{\Omega_-}=0\}.
\end{equation}
Since $\mathcal K 1\equiv -\frac12$, using the decomposition \eqref{eq:2.7} in the second equation of \eqref{eq:2.6} tested with $\mu\equiv 1$ yields the formula
\begin{equation}\label{eq:2.7b}
c=-\frac1{|\Gamma|}  \int_\Gamma \Big(\mathcal V\lambda +\smallfrac12(\gamma u_\star-\beta_0)-\mathcal K (\gamma u_\star-\beta_0)\Big).
\end{equation}

\subsection{Discretization with IPDG and BEM}

Consider now a conforming shape regular family of triangulations $\{\mathcal T_h\}$ of $\Omega_-$, made up of triangles/tetrahedra. Let $\mathcal E_h$ be the set of all edges/faces (in the two/three-dimensional cases) of elements of $\mathcal T_h$ and let $\mathcal E_h^\circ$ be the set of interior edges/faces. On each internal edge/face we assume that a fixed orientation of the normal vector has been chosen. Using this orientation, to each $e\in \mathcal E_h^\circ$ we can associate two elements $K_\pm\in \mathcal T_h$ such that $e=\overline K_+\cap \overline K_-$ and that the normal vector on $e$ points from $K_-$ to $K_+$. This allows us to introduce the following notation for jumps of traces and averages of normal derivatives
\begin{equation}\label{eq:2.8}
\jump{u}:=u|_{K_+}-u|_{K_-} \qquad \ave{\partial_\nu u}=\smallfrac12 (\partial_\nu u|_{K_+}+\partial_\nu u|_{K_-}).
\end{equation}
This notation coincides with the one used in \cite{OfRodSteTauSB} and differs from the one in the unified presentation of DG methods in \cite{ArnBreCocMar01} where the jump of a scalar magnitude is a vector pointing in the normal direction and the average of vector valued quantities is a scalar magnitude. For the purpose of our analysis (and for ease of comparison with \cite{OfRodSteTauSB}) the choice \eqref{eq:2.8} seems to be the adequate one. Note, however, that the algebraic expressions for the IPDG method are the same with the two choices of notation. The triangulation $\mathcal T_h$ creates a partition/triangulation $\Gamma_h$ of the boundary $\Gamma$.

The discrete spaces that are needed for the method are simply discontinuous piecewise polynomial functions:
\begin{eqnarray*}
V_h &:=& \{ v_h:\Omega \to \mathbb R\,:\, v_h|_K \in \mathbb P_k(K)\quad \forall K\in \mathcal T_h\},\\
\Lambda_h &:=& \{ \lambda_h :\Gamma \to \mathbb R\,:\, \lambda_h|_e\in \mathbb P_k(e)\quad \forall e \in \Gamma_h\}.
\end{eqnarray*}
For functions that are sufficiently smooth on each element, we can define the IPDG bilinear form:
\begin{eqnarray*}
a_{\mathrm{DG}}(u_h,v_h) &:=& (\nabla_h u_h,\nabla_h v_h)_{\Omega_-}-\sum_{e\in\mathcal E_h^\circ}\langle \ave{\partial_\nu u_h},\jump{v_h}\rangle_e\\
& &-\xi \sum_{e\in\mathcal E_h^\circ}\langle \ave{\partial_\nu v_h},\jump{u_h}\rangle_e + \sum_{e\in \mathcal E_h^\circ} \frac{\sigma_e}{h_e}\langle \jump{u_h},\jump{v_h}\rangle_e.
\end{eqnarray*}
Here, $\nabla_h$ is the gradient operator applied elementwise, $h_e$ is the length/diameter of $e\in \mathcal E_h$ and $\sigma_e>0$ is a constant associated to each edge/face.
Careful explanation of how to arrive at this bilinear form (and why), as well as of the three choices for $\xi\in \{-1,0,1\}$ can be found in \cite{ArnBreCocMar01}. A quantity we will need to control is
\[
\sigma_{\min}:= \min_{e\in \mathcal E_h^\circ} \sigma_e.
\]
The coupled IPDG-BEM scheme consists of the non-conforming Galerkin discretization of \eqref{eq:2.6} based on the subspaces $V_h$ and $\Lambda_h$, and on the substitution of the Dirichlet form $(\nabla u,\nabla v)_{\Omega_-}$ by $a_{\mathrm{DG}}(u,v)$. The method looks for $(u_h,\lambda_h)\in V_h\times \Lambda_h$ such that 
\begin{equation}\label{eq:2.9}
\left[ 
\begin{array}{lll}
\ds a_{\mathrm{DG}}(u_h,v_h) -\langle \lambda_h,v_h\rangle_\Gamma 
&\ds=(f,v_h)_{\Omega_-}+\langle\beta_1, v_h\rangle_\Gamma &  \forall v_h\in V_h,\\
\ds \langle \mu_h,\smallfrac12 u_h-\mathcal K u_h\rangle_\Gamma +\langle \mu_h,\mathcal V \lambda_h\rangle_\Gamma &=\langle\mu_h,\smallfrac12\beta_0-\mathcal K\beta_0\rangle_\Gamma & \forall \mu_h\in \Lambda_h.
\end{array}
\right.
\end{equation}
All occurrences of angled brackets in \eqref{eq:2.9} correspond to $L^2(\Gamma)$ inner products (they are not duality products any longer). In comparison with \eqref{eq:2.6}, the trace operator has been eliminated, since now $u_h$ and $v_h$ are piecewise smooth functions that can be restricted to the boundary, although not with the global trace operator. The additional regularity assumed for $\beta_1$ (for the sake of existence of solution to the original problem, only $\beta_1\in H^{-1/2}(\Gamma)$ is needed) is justified by its occurrence in the right-hand side of \eqref{eq:2.9}. The operator $\mathcal K$ acts on the restriction of $u_h$ to the boundary. Since $\mathcal K:L^2(\Gamma)\to L^2(\Gamma)$ is bounded (see \cite{McL00}), this term poses no problem from the practical point of view.

Algorithmic aspects of this formulation are treated in \cite{OfRodSteTauSB}. Let us just point out a detail referred to the two-dimensional case that is not covered in \cite{OfRodSteTauSB}. If we test the first equation of \eqref{eq:2.9} with $v_h \equiv 1 \in V_h$, it follows that $\lambda_h \in H^{-1/2}_0(\Gamma)$. This allows for the discrete exterior solution
\begin{equation}\label{eq:2.10}
u_{+,h}:= \mathcal D (u_h-\beta_0)-\mathcal S\lambda_h
\end{equation}
to have the right $\mathcal O(1/r)$ decaying behavior at infinity. For some future theoretical considerations, it will be convenient to consider the space
\[
\Lambda_h^{(0)}= \left\{ \begin{array}{ll} \Lambda_h\cap H^{-1/2}_0(\Gamma), & \mbox{if $d=2$},\\
\Lambda_h, & \mbox{if $d=3$}.\end{array}\right.
\]

Also, parallel to \eqref{eq:2.7}, and only for analytical purposes, it will be convenient to consider the decomposition
\begin{equation}\label{eq:2.11}
u_h=u_{h,\star}+c_h \qquad c_h \in \mathbb P_0(\Omega_-)\qquad u_{h,\star}\in V_h^\star:=V_h\cap L^2_\star(\Omega_-).
\end{equation}
Note that
\begin{equation}\label{eq:2.12}
c_h=-\frac1{|\Gamma|}  \int_\Gamma \Big(\mathcal V\lambda_h +\smallfrac12(u_{h,\star}-\beta_0)-\mathcal K (u_{h,\star}-\beta_0)\Big).
\end{equation}

\subsection{Main results}

In order to deal in a simpler way with the analysis of the method \eqref{eq:2.9}, we introduce a global bilinear form:
\begin{eqnarray*}
B((u_h,\lambda_h),(v_h,\mu_h))&:=& a_{\mathrm{DG}}(u_h,v_h)-\langle \lambda_h,v_h\rangle_\Gamma \\
& & +\langle \smallfrac12 u_h-\mathcal K u_h,\mu_h\rangle_\Gamma + \langle \mathcal V \lambda_h,\mu_h\rangle_\Gamma.
\end{eqnarray*}
In principle, this bilinear form is restricted to discrete elements in both components. With some abuse of notation, we will allow the exact solution of \eqref{eq:2.6} to be placed in the first component of this bilinear form. This can be done assuming some additional regularity for this solution and recuperating trace operators for the restrictions of $u$ to $\Gamma$. An important property of this method is its consistency:
\begin{equation}\label{eq:2.13}
B((u_h,\lambda_h),(v_h,\mu_h))=B((u,\lambda),(v_h,\mu_h))\qquad \forall (v_h,\mu_h)\in V_h\times \Lambda_h.
\end{equation}
For analytical purposes we now need to restrict the kind of meshes.We assume that there is a neighborhood of $\Gamma$ where the triangulation $\mathcal T_h$ is quasi-uniform. This fact will be needed for the use of an inverse inequality (see \eqref{eq:3.11} below) on the discrete space $\Lambda_h$. This might be restrictive in some practical applications when $\beta_0$ and $\beta_1$ do not vanish. Nevertheless, there are situations when the transmission problem \eqref{eq:2.1} (or equivalently \eqref{eq:2.6}) originates from cutting off the computational domain for a problem in free space: in this situation the boundary $\Gamma$ can be placed apart from the support of the data function $f$ and $u$ is continuous across $\Gamma$ (that is $\beta_0\equiv 0$ and $\beta_1\equiv 0$); for that case, the local quasi-uniformity of the triangulation might not be very restrictive. It has to be emphasized though, that at present the inverse inequality plays an important role in the coercivity estimates and cannot be easily removed.

The analysis starts with a quasi-coercivity estimate in terms of a discrete seminorm and the bilinear form induced by the operator $\mathcal V$ (see \eqref{eq:3.14}). For convenience, we write:
\[
|u_h|_h^2:=\sum_{e\in \mathcal E_h^\circ} h_e^{-1} \| \jump{u_h}\|_e^2 
\]
The detailed proof of this first result requires the introduction of several elements and the proof of some preliminary estimates. This will be dealt with in Section \ref{sec:3}. The proof of the theorem itself will be given in Section \ref{sec:3.4}.

\begin{theorem}[Coercivity]\label{th:2.1}
There exists a constant $\sigma_0$ depending on characteristics of the mesh (shape-regularity, local quasi-uniformity near $\Gamma$) and the polynomial degree $k$ such that if $\sigma_{\min}\ge \sigma_0$, then
\[
B((u_h,\lambda_h),(u_h,\lambda_h))\ge \smallfrac14\| \nabla_h u_h\|_{\Omega_-}^2+\smallfrac14 |u_h|_h^2+\smallfrac14\langle \lambda_h,\mathcal V\lambda_h\rangle, \qquad \forall (u_h,\lambda_h)\in V_h \times \Lambda_h^{(0)}.
\]
\end{theorem}

\begin{corollary}[Unique solvability]
Under the hypotheses of Theorem \ref{th:2.1}, the system \eqref{eq:2.9} has a unique solution.
\end{corollary}

\begin{proof}
Since \eqref{eq:2.9} can be equivalently written as a square system of linear equations, only the uniqueness part is needed. If $(u_h,\lambda_h)\in V_h\times \Lambda_h$ is a homogeneous solution of \eqref{eq:2.9}, then testing with $v_h\equiv 1$ shows that $\lambda_h \in \Lambda_h^{(0)}$. Theorem \ref{th:2.1} and \eqref{eq:3.14} show then that $\lambda_h\equiv 0$ and $u_h\equiv c$. The decomposition \eqref{eq:2.11}-\eqref{eq:2.12} shows finally that $u_h\equiv 0$.
\end{proof}

Convergence estimates will be only given for a solution of the highest regularity. For the interior field $u$, we will consider
\[
 H^{m}(\mathcal T_h):=\prod_{K\in \mathcal T_h} H^{m}(K),\qquad | u|_{m,\mathcal T_h}^2:=\sum_{K\in \mathcal T_h} | u|_{m,K}^2.
\]
For the boundary unknown $\lambda$, we consider the set of edges/faces $\{\Gamma_1,\ldots,\Gamma_M\}$ of the polygon/polyhedron $\Gamma$ and the broken (but not discrete) Sobolev spaces
\[
X^m(\Gamma):=\prod_{\ell=1}^M H^m(\Gamma_\ell),
\]
endowed with the product norms.

\begin{theorem}[Convergence]\label{th:2.3}
Assume that the solution of \eqref{eq:2.6} is in $H^{k+1}(\mathcal T_h)\times X^{k+1}(\Gamma)$. Then, the error of the method of \eqref{eq:2.9} can be bounded as
\[
\| u-u_h\|_{\Omega_-} + 
\| \nabla u-\nabla_h u_h\|_{\Omega_-}+|u_h|_h+\|\lambda-\lambda_h\|_{-1/2,\Gamma} \le C h^{k} \left( |u|_{k+1,\mathcal T_h}+|\lambda|_{k+1,\Gamma}\right).
\]
Moreover, for all $\mathbf x\in \Omega_+$,
\begin{equation}\label{eq:2.16}
| u_+(\mathbf x)-u_{+,h}(\mathbf x)| \le C h^{k} \| \Phi(\mathbf x,\punto)\|_{2,\Omega_-}  \left( |u|_{k+1,\mathcal T_h}+|\lambda|_{k+1,\Gamma}\right).
\end{equation}
\end{theorem}

 As will be made clear by the proof of this result (Section \ref{sec:4.3}), the estimates of Theorem \ref{th:2.3} can also be given in more local terms both for $u$ and $\lambda$, using local meshsizes. That part of the analysis is just related to approximation properties and we will not insist on it. Note that the bound $|u_h|_h=\mathcal O(h^k)$ takes into account how close $u_h$ gets to be continuous.

\section{Discrete coercivity}\label{sec:3}

\subsection{A geometric construction}

Part of the forthcoming analysis hinges on a class of rigid scaling arguments in order to deal with solenoidal vector fields and fractional order Sobolev spaces with different scalability properties.
We assume the existence of a finite set of reference configurations with the following characteristics:
\begin{itemize}
\item In the two-dimensional case $\widehat e:=(0,1)\times \{0\}$ and $\widehat K$ is a fixed isosceles triangle with $\widehat e$ as base and such that its two equal angles are less than half the minimum angle of the sequence of triangulations $\mathcal T_h$. By rotation, translation and scaling, we can place an isosceles triangle $K_e$, congruent to $\widehat K$, with base on every $e\in\mathcal E_h^\circ$ and such that $K_e\subset\Omega_-$. In order to have unified notation with the three-dimensional case we will write $\widehat{\mathcal K}:=\{\widehat K\}$.
\item In the three-dimensional case we take a set of pyramids $\widehat{\mathcal K}:=\{ \widehat K_1,\ldots,\widehat K_\ell\}$ with respective bases $\{ \widehat e_1,\ldots,\widehat e_\ell\}$, with $\widehat e_j\subset \mathbb R^2\times \{0\}$. We assume that every interior face of the triangulation $e\in\mathcal E_h^\circ$, with diameter $h_e$, can be extended to be the base of a pyramid $K_e \subset \Omega_-$, such that $h_e^{-1}K_e$ is a rigid motion of one of the elements of  $\widehat{\mathcal K}$.
\end{itemize} 
Given now any $e\in \mathcal E_h^\circ$, we consider an invertible affine map $F_e:\mathbb R^d\to \mathbb R^d$ of the form
\[
F_e(\mathbf x):=h_e \Theta_e \mathbf x+\mathbf b_e, \qquad \Theta_e^\top\Theta_e=\mathrm I_d,
\]
such that $F_e^{-1}(K_e)\in \widehat{\mathcal K}$ and $F_e^{-1}(e)$ is contained in the base of $F_e^{-1}(K_e)$. Note that in the two-dimensional case, this implies that $F_e^{-1}(e)=\widehat e$, whereas in the three-dimensional case  $F_e^{-1}(e)$ is a triangle contained in the base of $F_e^{-1}(K_e)$. By construction, there exists $C_\circ>0$ (in the two-dimensional case $C_\circ=1$) such that
\begin{equation}\label{eq:3.1}
\sum_{e\in \mathcal E_h^\circ}\| v\|_{K_e}^2\le C_\circ \|v\|_{\Omega_-}^2\qquad \forall v \in L^2(\Omega_-).
\end{equation}

\subsection{Fractional order Sobolev norms in reference configurations}

Given a bounded open set $\mathrm O \subset \mathbb R^d$ (or the closure of a bounded open set) and $s\in (0,1)$ we consider the space $H^s(\mathrm O)$, endowed with the Sobolev-Slobodetskij  norm
\begin{equation}\label{eq:3.2}
\| v\|_{s,\mathrm O}^2:=\| v\|_{\mathrm O}^2+|v|_{s,\mathrm O}^2, \qquad |v|_{s,\mathrm O}^2:=\int_{\mathrm O}\int_{\mathrm O}\frac{|v(\mathbf x)-v(\mathbf y)|^2}{|\mathbf x-\mathbf y|^{2s+d}}\mathrm d\mathbf x\mathrm d\mathbf y
\end{equation}
and the space $\widetilde H^s(\mathrm O)$, endowed with the norm
\begin{equation}\label{eq:3.3}
\| v\|^2_{s,\mathrm O,\sim}:=|v|_{s,\mathrm O}^2+\int_{\mathrm O}\frac{|v(\mathbf x)|^2}{\mathrm{dist}(\mathbf x,\partial\mathrm O)^{2s}}\mathrm d\mathbf x.
\end{equation}
The negative order spaces are defined by duality, considering the duality representation that pivots around $L^2(\mathrm O)$,
\[
\widetilde H^{-s}(\mathrm O):=H^s(\mathrm O)', \qquad H^{-s}(\mathrm O):=\widetilde H^s(\mathrm O)',
\]
and endowed with the dual norms:
\begin{equation}\label{eq:3.4}
\| v\|_{-s,\mathrm O}:=\sup_{u\in \widetilde H^s(\mathrm O)}\frac{(u,v)_{\mathrm O}}{\| u\|_{s,\mathrm O,\sim}}, \qquad
\| v\|_{-s,\mathrm O,\sim}:=\sup_{u\in H^s(\mathrm O)}\frac{(u,v)_{\mathrm O}}{\| u\|_{s,\mathrm O}}. 
\end{equation}
For open polyhedral surfaces (polygonal curves), we can consider the spaces $H^s(S)$ and $\widetilde H^s(S)$ for $-1<s<1$. For closed polyhedral surfaces $\Gamma$, the spaces $H^s(\Gamma)$ for $0<s<1$ are defined by parametrization and $H^{-s}(\Gamma):=H^s(\Gamma)'$, pivotal to $L^2(\Gamma)$.

Since for  $\varepsilon \in (0,1/2)$ the trace operator $H^{1-\varepsilon}(\widehat K) \to H^{1/2-\varepsilon}(\partial\widehat K)$ is bounded and surjective, it has a bounded right inverse $L_\varepsilon$ and we can bound, with a constant $C_\varepsilon^{\circled1}>0$ that depends only on $\varepsilon$,
\begin{equation}\label{eq:3.5}
\| \nabla (L_\varepsilon\varphi)\|_{-\varepsilon,\widehat K}\le C_\varepsilon^{\circled1}\|\varphi\|_{1/2-\varepsilon,\partial\widehat K} \qquad \forall \varphi \in  H^{1/2-\varepsilon}(\partial\widehat K), \quad \widehat K\in\widehat{\mathcal K}.
\end{equation}
Since $H^s(\widehat K)\cong \widetilde H^s(\widehat K)$ for $0<s<1/2$, with equivalent but different norms, for $\varepsilon\in (-1/2,1/2)$, we can identify constants $C_\varepsilon^{\circled2}>0$ depending only on $\varepsilon$ such that
\begin{equation}\label{eq:3.6}
\| v\|_{-\varepsilon,\widehat K,\sim}\le C_\varepsilon^{\circled2}\|v\|_{-\varepsilon,\widehat K} \qquad \forall v\in H^{-\varepsilon}(\widehat K),\qquad \widehat K \in \widehat{\mathcal K}.
\end{equation}
The third group of inequalities requires some additional work in the three-dimensional case.

\begin{lemma}\label{lemma:3.1} For $\varepsilon\in (0,1/2)$, there exists $C_\varepsilon^{\circled3}>0$ such that
\begin{equation}\label{eq:3.7}
\| \widetilde\varphi\|_{1/2-\varepsilon,\partial\widehat K}\le C_\varepsilon^{\circled3} \| \varphi\|_{1/2-\varepsilon,D,\sim} \qquad \forall \varphi \in \widetilde H^{1/2-\varepsilon}(D), \quad \widehat K \in \widehat{\mathcal K},
\end{equation}
where $D=\widehat e$ if $d=2$ and $D$ is any triangle contained in the base of $\widehat K$ if $d=3$, and $\widetilde\varphi$ is the extension by zero of $\varphi$.
\end{lemma}

\begin{proof}
The first part of the proof is only required for the three-dimensional case. Let $\widehat e$ be the base of $\widehat K$ and $D$ be any triangle contained in $\widehat e$. Then
\begin{eqnarray*}
|\widetilde\varphi|_{s,\widehat e}^2&=&|\varphi|_{s,D}^2 + 2 \int_D \left( \int_{\widehat e\setminus D} \frac{\mathrm d\mathbf x}{|\mathbf x-\mathbf y|^{2s+2}}\right) |\varphi(\mathbf y)|^2\mathrm d\mathbf y\\
&\le & |\varphi|_{s,D}^2 + 2 \int_D \left( 2\pi \int_{\mathrm{dist}(\mathbf y,\partial D)}^\infty \frac{\mathrm dr}{r^{2s+1}} \right)  |\varphi(\mathbf y)|^2\mathrm d\mathbf y\\
& \le &  |\varphi|_{s,D}^2 +\frac{2\pi}{s} \int_D \frac{|\varphi(\mathbf y)|^2}{\mathrm{dist}(\mathbf y,\partial D)^{2s}}\mathrm d\mathbf y,
\end{eqnarray*}
and therefore
\[
\| \widetilde\varphi\|_{s,\widehat e,\sim}^2 \le \left(\frac{2\pi}{s}+1\right) \| \varphi\|_{s,D,\sim}^2,
\]
where the key fact is that the constant does not depend on $D$.
In the second step, we apply that extension-by-zero is a bounded operator from $\widetilde H^{s}(\widehat e)$ to $H^s(\partial\widehat K)$ and take $s=1/2-\varepsilon$.
\end{proof}

\begin{lemma}\label{lemma:3.2}
For $\varepsilon \in (0,1/2)$, there exists $C_\varepsilon^{\circled4}>0$ such that
\begin{equation}\label{eq:3.8}
\| \mathbf v\cdot\mathbf n\|_{-1/2+\varepsilon,D}\le C_\varepsilon^{\circled4}\| \mathbf v\|_{\varepsilon,\widehat K}\quad \forall \mathbf v \in H^\varepsilon(\widehat K)^d \mbox{ such that } \mathrm{div}\,\mathbf v=0
\end{equation}
where $D=\widehat e$ if $d=2$ and $D$ is any triangle contained in the base of $\widehat K$ if $d=3$.
\end{lemma}

\begin{proof}
Let $\varphi\in \widetilde H^{1/2-\varepsilon}(D)$,  $\widetilde\varphi$ be the extension by zero of $\varphi$ to the rest of $\partial \widehat K$, and $u:=L_\varepsilon \widetilde\varphi$, where $L_\varepsilon$ is the trace lifting of \eqref{eq:3.5}. Since the hypotheses on $\mathbf v$ imply that $\mathbf v\in H(\mathrm{div},\widehat K)$, we can apply Green's Theorem and prove that
\begin{eqnarray*}
\langle \mathbf v\cdot\mathbf n,\varphi\rangle_D &=& \langle \mathbf v\cdot\mathbf n,\widetilde\varphi\rangle_{\partial\widehat K}=(\mathrm{div}\,\mathbf v,u)_{\widehat K}+(\mathbf v,\nabla u)_{\widehat K}\\
&\le & \| \mathbf v\|_{\varepsilon,\widehat K}\|\nabla u\|_{-\varepsilon,\widehat K,\sim}\le C_\varepsilon^{\circled2}
\| \mathbf v\|_{\varepsilon,\widehat K}\|\nabla u\|_{-\varepsilon,\widehat K}\\
&\le & C_\varepsilon^{\circled1}C_\varepsilon^{\circled2} \| \mathbf v\|_{\varepsilon,\widehat K}\|\widetilde\varphi\|_{1/2-\varepsilon,\partial\widehat K}\le  C_\varepsilon^{\circled1}C_\varepsilon^{\circled2} C_\varepsilon^{\circled3}
 \| \mathbf v\|_{\varepsilon,\widehat K}\| \varphi\|_{1/2-\varepsilon,D,\sim},
\end{eqnarray*}
where we have applied \eqref{eq:3.6}, \eqref{eq:3.6} and Lemma \ref{lemma:3.1}. The proof is now a direct consequence of the definition of the $H^{-1/2+\varepsilon}(D)$ norm.
\end{proof}

\subsection{Two key lemmas}

 We first divide the set of edges as
\begin{equation}\label{eq:3.9}
\mathcal E_h^{\mathrm{strip}}:=\{ e \in \mathcal E_h^\circ \,:\, \mathrm{dist}(e,\Gamma) \le C\} \qquad\mbox{and}\qquad \mathcal E_h^{\mathrm{int}}:=\mathcal E_h^\circ\setminus\mathcal E_h^{\mathrm{strip}}.
\end{equation}
If $C$ above is large enough (or the triangulation is refined enough), we can fit
\begin{equation}\label{eq:3.10}
\cup\{ K_e \,:\, e\in \mathcal E_h^{\mathrm{int}}\} \subset \overline\Omega_{\mathrm{int}} \subset \Omega_-.
\end{equation}
Since we are assuming that the triangulation is shape-regular and quasi-uniform near the boundary $\Gamma$, then the partition $\Gamma_h$ is quasi-uniform and the diameter of the elements of $\Gamma_h$ is equivalent to
\begin{equation}\label{eq:3.10b}
h_{\mathrm{strip}}:= \max\{ h_e \,:\, e \in \mathcal E_h^{\mathrm{strip}}\cup \Gamma_h\}.
\end{equation}
Therefore, we have an inverse inequality for elements of the discrete space $\Lambda_h$:
\begin{equation}\label{eq:3.11}
h_{\mathrm{strip}}^\varepsilon \|\lambda_h\|_{-1/2+\varepsilon,\Gamma}\le C^{\mathrm{inv}}_{\varepsilon}\|\lambda_h\|_{-1/2,\Gamma} \qquad \forall \lambda_h \in \Lambda_h, \quad \varepsilon \in [0,1/2].
\end{equation}
An additional ingredient for the proofs below is related to continuity of layer potentials, namely, we can choose $C_\varepsilon^{\circled5}$ for every $\varepsilon \in [0,1/2]$ such that
\begin{equation}\label{eq:3.12}
\| \nabla(\mathcal S\lambda)\|_{\varepsilon,\Omega_{\mathrm{int}}}\le C_\varepsilon^{\circled5}\|\lambda\|_{-1/2,\Gamma}\qquad\forall \lambda\in H^{-1/2}(\Gamma)
\end{equation}
and
\begin{equation}\label{eq:3.13}
\|\nabla(\mathcal S\lambda)\|_{\varepsilon,\Omega_-}\le C_\varepsilon^{\circled5}\|\lambda\|_{-1/2+\varepsilon,\Gamma}\qquad \forall \lambda\in H^{-1/2+\varepsilon}(\Gamma).
\end{equation}
Note that the interior regularity bound \eqref{eq:3.12} can be proved directly by bounding the fundamental solution, since the distance between $\Omega_{\mathrm{int}}$ and $\Gamma$ is positive. On the other hand, \eqref{eq:3.13} is a well-known regularity result of layer potentials (see \cite{McL00}).

\begin{lemma}\label{lemma:3.3}
For all $\varepsilon \in (0,1/2)$, there exists $C_\varepsilon^{\circled6}>0$ such that
\[
\sum_{e\in \mathcal E_h^\circ} h_e^{d-2} \| \partial_\nu (u^*\circ F_e)\|_{-1/2+\varepsilon,F_e^{-1}(e)}^2 \le C_\varepsilon^{\circled6}\|\nabla u^*\|_{\mathbb R^d}^2 \qquad \forall u^*=\mathcal S \lambda_h, \quad \lambda_h \in \Lambda_h^{(0)}.
\]
\end{lemma}

\begin{proof} Note that $\Delta u^*=0$ in $K_e$ for all $e$, since $K_e\subset \Omega_-$. Since $F_e$ is a dilation and a rigid transformation, it is clear that $\Delta (u^*\circ F_e)=0$ in $\widehat K$. Therefore, by Lemma \ref{lemma:3.2}, we can bound
\begin{equation}\label{eq:3.15}
 \| \partial_\nu (u^*\circ F_e)\|_{-1/2+\varepsilon,F_e^{-1}(e)}\le C_\varepsilon^{\circled4} \|\nabla (u^*\circ F_e)\|_{\varepsilon,\widehat K}\qquad\forall e \in \mathcal E_h^\circ.
\end{equation}
A simple change of variables proves that
\begin{eqnarray}\nonumber
\|\nabla (u^*\circ F_e)\|_{\varepsilon,\widehat K}^2 &=& \|\nabla (u^*\circ F_e)\|_{\widehat K}^2+|\nabla (u^*\circ F_e)|_{\varepsilon,\widehat K}^2 \\
&=& h_e^{2-d} \Big(\| \nabla u^*\|_{K_e}^2 + h_e^{2\varepsilon}|\nabla u^*|_{\varepsilon,K_e}^2\Big).
\label{eq:3.16}
\end{eqnarray}
Before adding over all edges, notice that the estimate \eqref{eq:3.1} and the definitions of the subsets of edges \eqref{eq:3.9} and property \eqref{eq:3.10} imply that
\[
\sum_{e\in \mathcal E_h^{\mathrm{int}}} |v|_{\varepsilon,K_e}^2 \le C_\circ | v|_{\varepsilon,\Omega_{\mathrm{int}}}^2
\quad\mbox{and}\quad \sum_{e\in \mathcal E_h^{\mathrm{strip}}} |v|_{\varepsilon,K_e}^2 \le C_\circ | v|_{\varepsilon,\Omega_-}^2.
\]
Hence, \eqref{eq:3.16}, the estimates \eqref{eq:3.12}-\eqref{eq:3.13} and the inverse inequality \eqref{eq:3.11} imply that
\begin{eqnarray}\nonumber
\sum_{e\in \mathcal E_h^\circ} h_e^{d-2}\|\nabla (u^*\circ F_e)\|_{\varepsilon,\widehat K}^2 &\le & C_\circ\Big( \|\nabla u^*\|_{\Omega_-}^2 + h^{2\varepsilon} |\nabla u^*|_{\varepsilon,\Omega_{\mathrm{int}}}^2+ h_{\mathrm{strip}}^{2\varepsilon}  |\nabla u^*|_{\varepsilon,\Omega_-}^2\Big)\\
&\le & C_\circ \Big( \|\nabla u^*\|_{\Omega_-}^2 + (C_\varepsilon^{\circled5})^2 ( h^{2\varepsilon}+C_{\mathrm{inv}}^{2-2\varepsilon}) \|\lambda_h\|_{-1/2,\Gamma}^2\Big).\qquad
\label{eq:3.17}
\end{eqnarray}
The result is then a straightforward consequence of \eqref{eq:3.15}, \eqref{eq:3.17} and \eqref{eq:3.14}.
\end{proof}

\begin{lemma}\label{lemma:3.4}
For all $\varepsilon \in (0,1/2)$, there exists $C_\varepsilon^{\circled7}>0$ such that
\[
\sum_{e\in \mathcal E_h^\circ} h_e^{d-2}\| \jump{u_h\circ F_e}\|_{1/2-\varepsilon,F_e^{-1}(e),\sim}^2\le C_\varepsilon^{\circled7}
|u_h|_h^2 \qquad \forall u_h \in V_h.
\]
\end{lemma}

\begin{proof}
Some preparations are first needed in the three-dimensional case.
Let $\widehat D$ be a fixed triangle in $\mathbb R^2\times \{0\}$ and let $\widehat D_e:=F^{-1}_e(e)$. The affine maps $G_e:\mathbb R^2\to \mathbb R^2$ that transform $\widehat D$ to $\widehat D_e$ are uniformly bounded and have uniformly bounded inverses, because the triangles $\widehat D_e$ are shape regular and their area is of order one. Therefore, we can bound
\begin{equation}\label{eq:3.18}
 \| \varphi\|_{s,\widehat D_e,\sim}\le B_s \| \varphi\circ G_e\|_{s,\widehat D,\sim} \qquad \forall \varphi \in H^s(\widehat D_e),\qquad 0 < s < 1.
\end{equation}
Using  \eqref{eq:3.18} with $s=1/2-\varepsilon$, a finite dimension argument on $\widehat D$, and the change of variables $G_e$,  we can prove that there exists $C_\varepsilon^{\circled7}$ such that
\begin{equation}\label{eq:3.19}
\| \hat u\|_{1/2+\varepsilon,\widehat D_e,\sim}^2\le C_\varepsilon^{\circled7} \| \hat u\|_{\widehat D_e}^2 \qquad \forall \hat u \in \mathbb P_k.
\end{equation}
Note that this result is straightforward in the two-dimensional case since then for all $e$, $\widehat D_e=\widehat e$. Finally
\[
 h_e^{d-2}\| \jump{u_h\circ F_e}\|_{1/2-\varepsilon,F_e^{-1}(e),\sim}^2 \le  h_e^{d-2}C_\varepsilon^{\circled7} \| \jump{u_h\circ F_e}\|_{\widehat D_e}^2=C_\varepsilon^{\circled7}h_e^{-1} \|\jump{u_h}\|_e^2.
\]
\end{proof}

\subsection{Proof of Theorem \ref{th:2.1}}\label{sec:3.4}

A standard argument in DG analysis (see \cite{ArnBreCocMar01} for example) can be used to prove that
\begin{equation}\label{eq:3.20}
a_{\mathrm{DG}}(u_h,u_h) \ge  \| \nabla_h u_h\|_{\Omega_-}^2 -|1+\xi| \Big( C_\star \delta \| \nabla_h u_h\|_{\Omega_-}^2+ (4\delta)^{-1} |u_h|_h^2\Big) + \sigma_{\min} |u_h|_h^2,
\end{equation}
where $C_\star$ is a positive constant that allows us to bound
\begin{equation}\label{eq:3.20b}
\sum_{K\in \mathcal T_h} \sum_{\mathcal E_h \in e\subset \partial K} 
h_e \| v_h\|_e^2 \le C_\star \| v_h\|_{\Omega_-}^2 \qquad \forall v_h \in V_h.
\end{equation}
Taking $\delta=1/(4C_\star |1+\xi|)$ in \eqref{eq:3.20} for $\xi\neq -1$, we can present all cases together with the estimate
\begin{equation}\label{eq:3.21}
a_{\mathrm{DG}}(u_h,u_h)  \ge \smallfrac34 \| \nabla_hu_h\|_{\Omega_-}^2 +(\sigma_{\min}-|1+\xi|^2 C_\star)|u_h|_h^2 \qquad \forall u_h \in V_h.
\end{equation}
We now turn our attention to the full bilinear form. Let $\lambda_h \in \Lambda_h$ (satisfying the additional constraint $\int_\Gamma \lambda_h=0$ in the two-dimensional case). Let then $u^*:=\mathcal S \lambda_h$. By well-known results on boundary integral operators it follows that
\begin{equation}
\smallfrac12\lambda_h-\mathcal K^t\lambda_h =-\partial_\nu^+ u^*, \qquad \lambda_h=\partial_\nu^-u^*-\partial_\nu^+u^*,\qquad \langle\mathcal V\lambda_h,\lambda_h\rangle_\Gamma=\| \nabla u^*\|_{\mathbb R^d}
\end{equation}
(recall \eqref{eq:3.14}). Therefore
\begin{eqnarray*}
B((u_h,\lambda_h),(u_h,\lambda_h)) &=& a_{\mathrm{DG}}(u_h,u_h) -\langle\lambda_h,u_h\rangle_\Gamma + \langle u_h,\smallfrac12\lambda_h-\mathcal K^t\lambda_h\rangle_\Gamma + \langle\mathcal V\lambda_h,\lambda_h\rangle_\Gamma\\
&=&
a_{\mathrm{DG}}(u_h,u_h)-\langle \partial_\nu^- u^*-\partial_\nu^+ u^*,u_h\rangle_\Gamma -\langle \partial_\nu^+u^*,u_h\rangle_\Gamma + \| \nabla u^*\|_{\mathbb R^d}^2\\
&=& a_{\mathrm{DG}}(u_h,u_h) -(\nabla u^*,\nabla_h u_h)_{\Omega_-}+ \| \nabla u^*\|_{\mathbb R^d}^2+
\sum_{e\in\mathcal E_h^\circ} \langle \ave{\partial_\nu u^*},\jump{u_h}\rangle_e,
\end{eqnarray*}
where in the last step we have applied Green's formula element-wise, taking advantage of the fact that $\Delta u^*=0$. 
Using \eqref{eq:3.21} and the fact that $\partial_\nu u^*$ is single-valued across inter-element faces, it follows that
\begin{eqnarray}\nonumber
B((u_h,\lambda_h),(u_h,\lambda_h)) &\ge & \smallfrac14 \| \nabla_h u_h\|_{\Omega_-}^2 +\smallfrac12 \| \nabla u^*\|_{\mathbb R^d}^2\\
\label{eq:3.23}
& & + (\sigma_{\min}-|1+\xi|^2 C_\star)|u_h|_h^2 + \sum_{e\in \mathcal E_h^\circ} \langle \partial_\nu u^*,\jump{u_h}\rangle_e.
\end{eqnarray}
Finally, since $F_e$ are rigid motions composed with dilations, we can easily apply Lemmas \ref{lemma:3.3} and \ref{lemma:3.4} to bound
\begin{eqnarray*}
\sum_{e\in \mathcal E_h^\circ} \langle \partial_\nu u^*,\jump{u_h}\rangle_e &=&  \sum_{e\in \mathcal E_h^\circ}h_e^{d-2}\langle \partial_\nu(u^*\circ F_e),\jump{u_h\circ F_e}\rangle_{F^{-1}_e(e)}\\
&\le &  \sum_{e\in \mathcal E_h^\circ}h_e^{d-2}\|\partial_\nu(u^*\circ F_e)\|_{-1/2+\varepsilon,F_e^{-1}(e)}\|\jump{u_h\circ F_e}\|_{1/2-\varepsilon,F^{-1}_e(e),\sim}\\
&\le & C_\varepsilon^{\circled6} \delta \| \nabla u^*\|_{\mathbb R^d}^2+ 
\frac{C_\varepsilon^{\circled7}}{4\delta}
|u_h|_h^2.
\end{eqnarray*}
Taking now $\delta=C_\varepsilon^{\circled6}/4$  and inserting this bound in \eqref{eq:3.23}, it follows that
\[
B((u_h,\lambda_h),(u_h,\lambda_h)) \ge \smallfrac14 \| \nabla_h u_h\|_{\Omega_-}^2 +\smallfrac14 \| \nabla u^*\|_{\mathbb R^d}^2+ (\sigma_{\min}-|1+\xi|^2 C_\star-C_\varepsilon^{\circled7}/C_\varepsilon^{\circled6})|u_h|_h^2.
\]
The result now follows by simply taking $\sigma_{\min}\ge \smallfrac14 +|1+\xi|^2 C_\star+C_\varepsilon^{\circled7}/C_\varepsilon^{\circled6}$ and using \eqref{eq:3.14}.

\section{Convergence analysis}\label{sec:4}

\subsection{Quasi-optimality}

We consider the following discrete seminorm, defined in $H^1(\mathcal T_h)\times H^{-1/2}(\Gamma)$
\[
\triple{(u,\lambda)}_h^2:=\|\nabla_h u\|_{\Omega_-}^2+|u|_h^2+\|\lambda\|_{-1/2,\Gamma}^2,
\]
as well as the stronger norm
\[
\triple{(u,\lambda)}_{h,\mathrm{st}}^2:=\triple{(u,\lambda)}_h^2+\sum_{e\in \Gamma_h} h_e^{-1} \|u\|_e^2+\sum_{e\in \Gamma_h} h_e^{-1} \|\lambda\|_e^2+\sum_{K\in\mathcal T_h} h_K^2 |u|_{2,K}^2,
\]
defined in $H^2(\mathcal T_h)\times L^2(\Gamma)$. 

\begin{proposition}[Continuity of the bilinear form]\label{prop:4.1} There exists $C_1>0$ such that
\[
|B((u,\lambda),(v_h,\mu_h))|\le C_1 \triple{(u,\lambda)}_{h,\mathrm{st}}\triple{(v_h,\mu_h)}_h \qquad \begin{array}{l}\forall (u,\lambda)\in H^2(\mathcal T_h)\times L^2(\Gamma),\\ \forall (v_h,\mu_h)\in V_h^\star\times \Lambda_h.\end{array}
\]
\end{proposition}

\begin{proof} Using a local trace inequality
\begin{equation}\label{eq:4.0}
\sum_{\mathcal E_h \ni e\subset \partial K} h_e \| v\|_e^2 \le h_K \| v\|_{\partial K}^2\le  C_{\star\star} \left( \| v\|_K^2 + h_K^2\|\nabla v\|_K^2\right)\qquad \forall v \in H^1(K),
\end{equation}
and applying \eqref{eq:3.20b} to the components of $\nabla_h v_h$,  it is simple to obtain the bound
\begin{eqnarray}\nonumber
|a_{\mathrm{DG}}(u,v_h)| & \le & \|\nabla_h u\|_{\Omega_-}\|\nabla_h v_h\|_{\Omega_-}+ C_{\star\star}^{1/2}
 \left( \| \nabla_h u\|_{\Omega_-}^2 + \sum_{K\in \mathcal T_h}h_K^2| u|_{2,K}^2\right)^{1/2}|v_h|_h\\
& & + C_\star^{1/2} | u|_h\|\nabla_h v_h\|_{\Omega_-} + \sigma_{\max} |u|_h|v_h|_h.\label{eq:4.1}
\end{eqnarray}
Also, using a discrete Poincar\'e-Friedrichs inequality (see \cite{Bre03} and \cite{Arn82}), we can bound
\begin{equation}\label{eq:4.2}
\| v_h\|_{\Omega_-} \le C_{\mathrm{PF}} \Big( \|\nabla_h v_h\|_{\Omega_-}^2+|v_h|_h^2\Big)^{1/2}\qquad \forall v_h \in V_h^\star,
\end{equation}
which, together with \eqref{eq:3.20b} allows us to bound
\begin{eqnarray}\nonumber
|\langle\lambda,v_h\rangle_\Gamma| & \le &\left( \sum_{e\in \Gamma_h} h_e^{-1} \|\lambda\|_e^2\right)^{1/2}\left( \sum_{e\in \Gamma_h} h_e \| v_h\|_e^2\right)^{1/2}\\
&\le& C_\star^{1/2}C_{\mathrm{PF}} \left( \sum_{e\in \Gamma_h} h_e^{-1} \|\lambda\|_e^2\right)^{1/2} \Big( \|\nabla_h v_h\|_{\Omega_-}^2+|v_h|_h^2\Big)^{1/2}.\label{eq:4.3}
\end{eqnarray}
The third term in the bilinear form is bounded using the inverse inequality \eqref{eq:3.11} (recall \eqref{eq:3.10b} for the definition of the local meshsize), proving that
\begin{eqnarray}\nonumber
|\langle\smallfrac12 u-\mathcal K u,\mu_h\rangle_\Gamma| &\le& \|\smallfrac12\mathcal I-\mathcal K\|_{L^2(\Gamma)\to L^2(\Gamma)}\|u\|_{\Gamma} \|\mu_h\|_{\Gamma}\\
&\le & C_{1/2}^{\mathrm{inv}}  \|\smallfrac12\mathcal I-\mathcal K\|_{L^2(\Gamma)\to L^2(\Gamma)} \left( \sum_{e\in \Gamma_h} h_e^{-1} \| u\|_e^2\right)^{1/2}\!\! \|\mu_h\|_{-1/2,\Gamma}.\label{eq:4.4}
\end{eqnarray}
The result is now a direct consequence of \eqref{eq:4.1}, \eqref{eq:4.3}, \eqref{eq:4.4}, the continuity of $\mathcal V:H^{-1/2}(\Gamma)\to H^{1/2}(\Gamma)$ and the definition of the broken norms.
\end{proof} 

Let $\Pi_h:L^2(\Omega_-)\to V_h$ and $P_h:L^2(\Gamma)\to \Lambda_h$  be the orthogonal projections onto $V_h$ and $\Lambda_h$ respectively. Since constant elements are in $V_h$, we can decompose (using \eqref{eq:2.7} and \eqref{eq:2.11})
\begin{equation}\label{eq:4.5}
u_h-\Pi_h u = u_{h,\star}-\Pi_h u_\star + c_h -c.
\end{equation}
Also, using the fact that constant functions are in $\Lambda_h$, we can show (see the first equations in \eqref{eq:2.6} and \eqref{eq:2.9}) that
\begin{equation}\label{eq:4.6}
\int_\Gamma P_h\lambda=\int_\Gamma \lambda = \int_{\Omega_-} f +\int_\Gamma\beta_1=\int_\Gamma \lambda_h.
\end{equation}
(All of them vanish in the two-dimensional case because of \eqref{eq:2.2}.) 

\begin{proposition}[Quasi-optimality]\label{prop:4.2}
Let $(u,\lambda)$ and $(u_h,\lambda_h)$ be the respective solutions of \eqref{eq:2.6} and \eqref{eq:2.9} and assume that $(u,\lambda)\in H^2(\mathcal T_h)\times L^2(\Gamma)$. Then there exists $C_2>0$ such that
\[
\triple{(u_h- u,\lambda_h-\lambda)}_h \le C_2 \triple{(u-\Pi_h u,\lambda-P_h\lambda)}_{h,\mathrm{st}}.
\]
\end{proposition}

\begin{proof}
Note that by \eqref{eq:4.6}, $\lambda_h-P_h \lambda \in \Lambda_h^{(0)}$. We can then apply Theorem \ref{th:2.1} and \eqref{eq:3.14} to bound
\begin{eqnarray}\nonumber
& & \hspace{-1cm}\triple{(u_h-\Pi_h u,\lambda_h-P_h\lambda)}_h^2 = \triple{(u_{h,\star}-\Pi_h u_\star,\lambda_h-P_h\lambda)}_h^2 \\
\nonumber
& & \le 4 \max\{1,C_\Gamma\} B((u_{h,\star}-\Pi_h u_\star,\lambda_h-P_h\lambda),(u_{h,\star}-\Pi_h u_\star,\lambda_h-P_h\lambda))\\
\label{eq:4.7}
& & =4 \max\{1,C_\Gamma\} B((u-\Pi_h u,\lambda_h-P_h\lambda),(u_{h,\star}-\Pi_h u_\star,\lambda_h-P_h\lambda)),
\end{eqnarray}
by the consistency of the method \eqref{eq:2.13}, the decomposition \eqref{eq:4.5} and the fact that
\[
B((c-c_h,0),(u_{h,\star}-\Pi_h u_\star,\lambda_h-P_h\lambda))=(c-c_h)\int_\Gamma (\lambda_h-P_h\lambda)=0
\]
(see \eqref{eq:4.6} and recall that $\mathcal K1\equiv -\frac12$). Since $u_{h,\star}-\Pi_h u_\star\in V_h^\star$, we can apply Proposition \ref{prop:4.1} and the result is a direct consequence of \eqref{eq:4.7}.
\end{proof}

\subsection{Estimate of the interior average}

Because we are dealing with the pure transmission problem (no boundary conditions), the discrete seminorm does not take into account the average of the component $u$ of the solution in $\Omega_-$. This is very clear from the coercivity estimate (Theorem \ref{th:2.1}), which is written in terms of that seminorm and, as such, is not able to estimate the error
\[
|c-c_h| =\frac1{|\Omega_-|}\left| \int_{\Omega_-} (u-u_h)\right|
\]
(see decompositions \eqref{eq:2.7} and \eqref{eq:2.11}). We start the analysis of this term with a lemma that combines some of the arguments of the previous sections.

\begin{lemma}\label{lemma:4.3} Let $\xi:=\frac12-\mathcal K^t 1$. Then there exists $C_3>0$ such that
\begin{equation}
|\langle \xi,v_h\rangle_\Gamma| \le C_3 \Big( \|\nabla_h v_h\|_{\Omega_-}^2+|v_h|_h^2\Big)^{1/2}\qquad \forall v_h \in V_h^\star.
\end{equation}
\end{lemma}

\begin{proof}
Note that $\xi=1-\partial_\nu^- \mathcal S 1$. We can then bound
\begin{equation}\label{eq:4.9}
|\langle \partial_\nu^- \mathcal S 1,v_h\rangle_\Gamma|\le C_4 \Big( \|\nabla_h v_h\|_{\Omega_-}^2+|v_h|_h^2\Big)^{1/2}
\qquad \forall v_h \in V_h,
\end{equation}
by proceeding as in the proof of Theorem \ref{th:2.1} (see Section \ref{sec:3.4}). Note that when we aim to apply Lemma \ref{lemma:3.3}, in the two-dimensional case $\lambda_h \equiv 1 \not\in \Lambda_h^{(0)}$. However, we can still bound as in \eqref{eq:3.17}
\[
\sum_{e\in \mathcal E_h^\circ} h_e^{d-2}\|\nabla ((\mathcal S 1)\circ F_e)\|_{\varepsilon,\widehat K}^2 \le C_\circ\Big( \|\nabla (\mathcal S 1)\|_{\Omega_-}^2 + h^{2\varepsilon} |\nabla (\mathcal S 1)|_{\varepsilon,\Omega_-}^2\Big),
\]
which is enough for our current needs. In order to estimate the remaining term, we consider the solution of
the interior Neumann problem
\[
-\Delta \widetilde u \equiv c \quad \mbox{in $\Omega^-$},\qquad \partial_\nu \widetilde u \equiv 1, \qquad \mbox{with } c:=\frac{|\Gamma|}{|\Omega_-|},
\]
and note that classical Sobolev regularity results can be invoked to show that $\widetilde u\in H^{3/2+\varepsilon}(\Omega_-)$ for some $\varepsilon>0$. We then apply integration by parts element-by-element to obtain
\begin{equation}\label{eq:4.10}
\langle 1,v_h\rangle_\Gamma = -\sum_{e\in \mathcal E_h^\circ} \langle\partial_\nu \widetilde u,\jump{v_h}\rangle_e + (\nabla \widetilde u,\nabla_h v_h)_{\Omega_-}- c (1,v_h)_{\Omega_-}.
\end{equation}
The first term is bounded using
\begin{eqnarray}\nonumber
\Big| -\sum_{e\in \mathcal E_h^\circ} \langle\partial_\nu \widetilde u,\jump{v_h}\rangle_e\Big| & \le &\Big( \sum_{e\in \mathcal E_h^\circ} h_e \| \partial_\nu \widetilde u\|_e^2 \Big)^{1/2}|v_h|_h\le 
\Big(\sum_{K\in \mathcal T_h} h_K \| \nabla \widetilde u\|_{\partial K}^2 \Big)^{1/2}|v_h|_h\\
&\le & C_\varepsilon^{\circled8}\Big( \sum_{K\in \mathcal T_h}  \| \nabla \widetilde u\|_{K}^2 + h_K^{2\varepsilon} | \nabla \widetilde u|_{\varepsilon,K}^2\Big)^{1/2}|v_h|_h,
\label{eq:4.11}
\end{eqnarray}
where the last inequality follows from a scaling argument and the trace inequality in a fixed reference element.
Since $v_h \in L^2_\star(\Omega_-)$, the result follows from the inequalities \eqref{eq:4.9}, \eqref{eq:4.10} and \eqref{eq:4.11}.
\end{proof}

\begin{proposition}\label{prop:4.4} There exists $C_5>0$ such that
\[
|c-c_h| \le C  \triple{(u-\Pi_h u,\lambda-P_h\lambda)}_{h,\mathrm{st}}
\]
\end{proposition}

\begin{proof}
Using the integral formulas for $c$ and $c_h$ (that is, \eqref{eq:2.7b} and \eqref{eq:2.12}), we can easily bound
\begin{eqnarray*}
|\Gamma| \, |c-c_h| & \le & \| \mathcal V (\lambda-\lambda_h)\|_{1/2,\Gamma} + \| \xi\|_\Gamma \| \gamma u_\star-\Pi_h u_\star\|_\Gamma + |\langle \xi,u_{\star,h}-\Pi_h u_\star\rangle_\Gamma|\\
&\le & \|\mathcal V\|_{H^{-1/2}(\Gamma) \to H^{1/2}(\Gamma)} \|\lambda-\lambda_h\|_{-1/2,\Gamma} + 
 \| \xi\|_\Gamma \| \gamma u-\Pi_h u\|_\Gamma + |\langle \xi,u_{\star,h}-\Pi_h u_\star\rangle_\Gamma|
\end{eqnarray*}
where $\xi:=\frac12-\mathcal K^t 1$. Note that $u_{h,\star}-\Pi_h u_\star\in V_h^\star$ and we can thus apply Lemma \ref{lemma:4.3} to bound
\[
 |\langle \xi,u_{\star,h}-\Pi_h u_\star\rangle_\Gamma|\le C_\xi \triple{(u_{\star,h}-\Pi_h u_\star,\lambda_h-P_h\lambda)}_h= \triple{(u_{h}-\Pi_h u,\lambda_h-P_h\lambda)}_h.
\]
The result is now a consequence of the quasi-optimality bound (Proposition \ref{prop:4.2}) and the definition of the discrete norms.
\end{proof}

\subsection{Proof of Theorem \ref{th:2.3}}\label{sec:4.3}

By simple approximation properties (using scaling arguments and the Bramble-Hilbert lemma), we can bound
\begin{equation}\label{eq:4.13}
\triple{(u-\Pi_h u,\lambda-P_h \lambda)}_{h,\mathrm{st}} \le C_6 \Big( \sum_{K\in \mathcal T_h} h_K^{2k} | u|_{k+1,K}^2 + \sum_{e\in \Gamma_h} h_e^{2k+1} |\lambda|_{k+1,e}^2\Big)^{1/2}
\end{equation}
This inequality and Proposition \ref{prop:4.2} provide the bound
\[
\| \nabla u-\nabla_h u_h\|_{\Omega_-} +|u_h|_h +\|\lambda-\lambda_h\|_{-1/2,\Gamma} \le  C_7 \Big( \sum_{K\in \mathcal T_h} h_K^{2k} | u|_{k+1,K}^2 + \sum_{e\in \Gamma_h} h_e^{2k+1} |\lambda|_{k+1,e}^2\Big)^{1/2}.
\]
For the bound of the $L^2(\Omega_-)$ error we use the decompositions \eqref{eq:2.7} and \eqref{eq:2.11} and the discrete Poincar\'e-Friedrichs inequality \eqref{eq:4.2} to prove that
\begin{eqnarray*}
\| \Pi_h u-u_h\|_{\Omega_-} & \le & |c-c_h|\,|\Omega_-|^{1/2}+ \| \Pi_h u_\star-u_{h,\star}\|_{\Omega_-}\\
&\le &  |c-c_h|\,|\Omega_-|^{1/2} + C_{\mathrm{PF}} \triple{(\Pi_h u_\star-u_{h,\star},0)}_h
\end{eqnarray*}
Propositions \ref{prop:4.2} and \ref{prop:4.4}, and the approximation estimate \eqref{eq:4.13} then provide the $L^2(\Omega_-)$-estimate. Note that this estimate is suboptimal (some of the IP methods included in this analysis have no superconvergence properties), but that it is needed as a complement of previously given error bounds, since  $\triple{\punto}_h$ is only a seminorm.

Subtracting the representation formula \eqref{eq:2.4} from the discrete approximation for the exterior solution \eqref{eq:2.10}, and using the definition of the layer potentials given in Section \ref{sec:2.1}, we can bound
\begin{eqnarray}\nonumber
| u_{+,h}(\mathbf x)-u_+(\mathbf x)| & \le & \| \Phi(\mathbf x,\punto)\|_{1/2,\Gamma} \|\lambda-\lambda_h\|_{-1/2,\Gamma} + \| \Phi(\mathbf x,\punto)\|_\Gamma \| \gamma u-\Pi_h u\|_\Gamma\\
& & +|\langle \partial_\nu \Phi(\mathbf x,\punto),\Pi_h u-u_h\rangle_\Gamma|.\label{eq:4.14}
\end{eqnarray}
For fixed $\mathbf x \in \Omega_+$, $\Phi(\mathbf x,\punto)\in \mathcal C^\infty (\overline\Omega_-)$ is harmonic. Using element-by-element integration by parts and a scaling argument \eqref{eq:4.0}, we can write
\begin{eqnarray*}
|\langle \partial_\nu \Phi(\mathbf x,\punto),v_h\rangle_\Gamma|&=&\Big|-\sum_{e\in \mathcal E_h^\circ} \langle \partial_\nu \Phi(\mathbf x,\punto),\jump{v_h}\rangle_e+(\nabla \Phi(\mathbf x,\punto),\nabla_h v_h)_{\Omega_-}\Big|\\
&\le & \Big( \sum_{K\in\mathcal T_h} h_K \|\nabla \Phi(\mathbf x,\punto)\|_{\partial K}^2\Big)^{1/2} | v_h|_h+\|\nabla \Phi(\mathbf x,\punto)\|_{\Omega_-}\| \nabla_h v_h\|_{\Omega_-}\\
&\le & \max\{C_{\star\star}^{1/2},\,h^2 C_{\star\star}^{1/2}+1\} \|\Phi(\mathbf x,\punto)\|_{2,\Omega_-}\triple{(v_h,0)}_h\qquad \forall v_h \in V_h.
\end{eqnarray*}
The bound \eqref{eq:2.16} is then a consequence of this latter inequality applied to $v_h=\Pi_h u-u_h$, \eqref{eq:4.14}
and the above estimates for the error.

\bibliographystyle{abbrv}
\bibliography{RefsBEMFEM}

\end{document}